\documentclass[11pt]{article}
\usepackage[arrow, matrix]{xy}
\usepackage{amsmath}
\usepackage{amssymb}
\usepackage{amsthm}
\usepackage{graphicx}
\usepackage[mathscr]{eucal}
\usepackage{color}
\theoremstyle{definition}
\newtheorem{definition}{Definition}
\newtheorem{theorem}[definition]{Theorem}

\theoremstyle{remark}
\newtheorem{remark}[definition]{Remark}
\newtheorem{example}[definition]{Example}
\newcounter{enumctr}

\newcommand{\N}{\mathbb{N}}
\newcommand{\R}{\mathbb{R}}

\renewcommand{\phi}{\varphi}

\begin{document}
\title{\vspace*{-10mm}
Stability of fractional-order nonlinear systems by Lyapunov direct method}
\author{H.T. Tuan\footnote{\tt httuan@math.ac.vn, \rm Institute of Mathematics, Vietnam Academy of Science and Technology, 18 Hoang Quoc Viet, 10307 Ha Noi, Viet Nam},
\;\;Hieu Trinh\footnote{\tt hieu.trinh@deakin.edu.au, \rm School of Engineering, Deakin University, Geelong, VIC 3217, Australia}
}
\date{}
\maketitle
\abstract{In this paper, by using a characterization of functions having fractional derivative and properties of positive solutions to a Volterra integral equation, we propose a rigorous fractional Lyapunov function candidate method to analyze the stability of fractional-order nonlinear systems.  First, we prove an inequality concerning the fractional derivatives of convex Lyapunov functions without the assumption of the existence of derivative of pseudo-states. Second, we establish fractional Lyapunov functions to fractional-order systems without the assumption of the global existence of solutions. Our theorems fill the gaps and strengthen results in some existing papers.
}
\section{Introduction}
Fractional differential equations have attracted increasing interest in the last decade due to the
fact that many mathematical problems in science and engineering can be modeled by fractional
differential equations. For more details on applications of fractional differential equations, we refer the interested reader to the monographs \cite{Babdyopadhyay}, \cite{Oldham}, \cite{Samko} and the references therein.

One of the most important problems in the qualitative theory of fractional differential equations is stability theory. Following Lyapunov's seminal 1892 thesis, these two methods are expected to also work for fractional differential equations:
\begin{itemize}
\item[$\bullet$] Lyapunov's first method: the method of linearization of the nonlinear equation along an
orbit, and the transfer of asymptotic stability from the linear
to the nonlinear equation; and
\item[$\bullet$] Lyapunov's second method: the method of Lyapunov candidate functions, i.e. of scalar functions
on the state space such that their energy decreases along orbits.
\end{itemize}

Recently, in \cite{Cong_3} and \cite{Tuan_1}, Cong {\textit{ et. al.}} fully developed Lyapunov's first method for fractional-order nonlinear systems. On the other hand, although several results on the Lyapunov's second method for fractional-order nonlinear systems have been published, the development of this theory is still in its infancy and requires further investigation. One of the reasons for this might be that computation and estimation of fractional derivatives of Lyapunov candidate functions are very complicated due to the fact that the well-known Leibniz rule does not hold true for such derivatives.

To the best of our knowledge, the first valuable contribution in the theory of fractional Lyapunov functions is the paper \cite{Li}. The method in \cite{Li} became applicable after effective fractional derivative inequalities were established, see e.g. \cite[Inequalities (6)\;\& (16)]{Camacho}, \cite[Inequality (24)]{Camacho_1}, and \cite[Inequality (10)]{Chen}. In this direction, we recommend the papers \cite[Theorems 2\;\& 3]{Yunquan}, \cite[Theorems 2\;\& 3]{Chen}, \cite[Theorems 3.1 \& 3.3]{Liu}, and \cite[Example 1]{Fernadez}. However, there are some unavoidable shortcomings of this approach such as:
\begin{itemize}
\item[$\bullet$] Assumption of the global existence of solutions to fractional-order nonlinear systems, see e.g. \cite{Li}, \cite{Camacho}, and \cite{Chen}.
\item[$\bullet$] The derivative of the solutions are required for the proof of the involved fractional derivative inequalities in \cite[Inequalities (6)\;\& (16)]{Camacho}, \cite[Inequality (24)]{Camacho_1}, and \cite[Inequality (10)]{Chen}.
\end{itemize}

Following another approach using the fractional derivative of the Lyapunov candidate function along the vector field, Lakshmikantham, Leela and Devi \cite[Theorem 4.3.2, pp. 100--101]{Lak} also attempted to prove a Lyapunov sufficient condition for fractional differential equations. However, confusion on the locality of solutions to fractional systems makes their proof incomplete, see \cite[pp. 101, lines 4--5]{Lak} (note that the solution to equation (4.2.1) in \cite[pp. 96]{Lak} starts from $t_0$ and its solution which starts from $t_1$ are different).

Motivated by the aforementioned observations, in this paper we focus on proposing a rigorous method of Lyapunov candidate functions which is suitable for fractional-order nonlinear systems. Specifically, we establish fractional Lyapunov functions without the assumption of the global existence of solutions to fractional-order nonlinear systems. We also do not require the condition on the existence of derivative to pseudo-states in the inequality concerning the fractional derivatives of convex Lyapunov functions. The rest of our paper is organized as follows. Section 2 is devoted to recalling some notations and results about fractional calculus. In Section 3, we formulate the main result which concerns the stability of the trivial solution to fractional-order systems based on designing an effective Lyapunov candidate function.

To conclude this introductory section, we introduce some notations which are used throughout the paper. Denote by $\N$ the set of nature numbers, by $\R$ and $\R_{+}$ the set of real numbers and non-negative numbers, respectively. For some arbitrary positive constant $d$, let $\R^d$ be the $d$-dimensional Euclidean space with the scalar $\langle \cdot,\cdot \rangle$ and the norm $\|\cdot\|$. In $\R^d$, let $B_r(0)$ be the closed ball with the center at the origin and the radius $r>0$. For some $T>0$, denote by $C([0,T],\R^d)$ the space of continuous functions $x:[0,T]\rightarrow \R^d$. Finally, for $\alpha \in (0,1]$, we mean $\mathcal{H}^\alpha([0, T],\R^d)$ the standard H\"older space consisting of
functions $v\in C([0, T],\R^d)$ such that
\[
\|v\|_{\mathcal{H}^\alpha}:=\max_{0\leq t\leq T}\|v(t)\|+\sup_{0\leq s<t\leq T}\frac{\|v(t)-v(s)\|}{(t-s)^\alpha}<\infty
\]
and by $\mathcal{H}^\alpha_0([0,T],\R^d)$ the closed subspace of $\mathcal{H}^\alpha([0, T],\R^d)$ consisting of functions $v \in \mathcal{H}^\alpha([0, T],\R^d)$ such that
\[
\sup_{0\leq s<t\leq T,t-s\leq \varepsilon}\frac{\|v(t)-v(s)\|}{(t-s)^\alpha}\to 0\quad \text{as}\;\varepsilon\to 0.
\]
\section{Preliminaries}
We recall briefly a framework of fractional calculus and fractional differential equations.

Let $\alpha\in (0,1)$, $[0,T]\subset \R$ and $x:[0,T]\rightarrow \R$ be a measurable function such that $x\in L^1([0,T])$, i.e.\ $\int_0^T|x(\tau)|\;d\tau<\infty$. Then, the Riemann--Liouville integral of order $\alpha$ is defined by
\[
I_{0+}^{\alpha}x(t):=\frac{1}{\Gamma(\alpha)}\int_0^t(t-\tau)^{\alpha-1}x(\tau)\;d\tau\quad \hbox{ for } t\in (0,T],
\]
where the Gamma function $\Gamma:(0,\infty)\rightarrow \R$ is defined as
\[
\Gamma(\alpha):=\int_0^\infty \tau^{\alpha-1}\exp(-\tau)\;d\tau,
\]
see e.g., Diethelm \cite{Kai}. The corresponding Riemann--Liouville fractional derivative of order $\alpha$ is given by
\[
^{R}D_{0+}^\alpha x(t):=(D I_{0+}^{1-\alpha}x)(t) \quad\forall t\in (0,T],
\]
where $D=\frac{d}{dt}$ is the usual derivative. On the other hand, the \emph{Caputo fractional derivative } $^{C\!}D_{0+}^\alpha x$ of $x$ is defined by
\[
^{C\!}D_{0+}^\alpha x(t):=^{R\!}D_{0+}^\alpha(x(t)-x(0)),\qquad \hbox{ for } t\in (0,T],
\]
see \cite[Definition 3.2, pp. 50]{Kai}. The Caputo fractional derivative of a $d$-dimensional vector function $x(t)=(x_1(t),\dots,x_d(t))^{T}$ is defined component-wise as $$^{C\!}D^\alpha_{0+}x(t)=(^{C\!}D^\alpha_{0+}x_1(t),\dots,^{C\!}D^\alpha_{0+}x_d(t))^{T}.$$

Denote by $I^\alpha_{0+}C([0,T],\R^d)$ the space of functions $\varphi:[0,T]\rightarrow \R^d$ such that there exists a function $\psi\in C([0,T],\R^d)$ satisfying $\varphi=I^\alpha_{0+}\psi$. The following result gives a characterization of functions having Caputo fractional derivative.
\begin{theorem}\label{Theorem1}
For $\alpha\in (0,1)$ and a function $v\in C([0,T],\R^d)$, the following conditions (i), (ii), (iii)  are equivalent:
\begin{itemize}
\item[(i)] the fractional derivative $^{C}D^\alpha_{0+}v\in C([0,T],\R^d)$ exists;
\item[(ii)] a finite limit $\lim_{t\to 0}\frac{v(t)-v(0)}{t^\alpha}:=\gamma$ exists, and
\[\sup_{0<t\leq T}\left\|\int_{\theta t}^t\frac{v(t)-v(\tau)}{(t-\tau)^{\alpha+1}}\;d\tau\right\|\to 0\quad \text{as}\; \theta\to 1;\]
\item[(iii)] $v$ has the structure $v-v(0)=t^\alpha \gamma+v_0$, where $\gamma$ is a constant vector, $v_0\in \mathcal{H}_0^\alpha ([0,T],\R^d)$, and $\int_0^t (t-\tau)^{-\alpha-1} (v(t)-v(\tau))d\tau=:w(t)$ converges for every $t\in (0,T]$ defining a function $w\in C((0,T],\R^d)$ which has a finite limit $\lim_{t\to 0}w(t)=:w(0)$.
\end{itemize}
For $v\in C([0,T],\R^d)$ having fractional derivative $^{C}D^\alpha_{0+}v\in C([0,T],\R^d)$, it holds $^{C}D^\alpha_{0+}v(0)=\Gamma(\alpha+1)\gamma$, and
\begin{align*}
^{C}D^\alpha_{0+}v(t)&=\frac{v(t)-v(0)}{\Gamma(1-\alpha)t^\alpha}\\
&\hspace{1cm}+\frac{\alpha}{\Gamma(1-\alpha)}\int_0^t \frac{v(t)-v(\tau)}{(t-\tau)^{\alpha+1}}d\tau,\quad 0<t\leq T.
\end{align*}
\end{theorem}
\begin{proof}
See \cite[Theorem 5.2, pp. 475]{Vainikko}.
\end{proof}

Let $D\subset \R^d$ is an open set and $0\in D$. In this paper, we consider the following equation with the fractional order $\alpha\in (0,1)$:
\begin{equation}\label{Eq_Ex}
^{C}D^\alpha_{0+}x(t)=f(x(t)),\quad \text{for}\; t\in (0,\infty),
\end{equation}
where $f:D\rightarrow \R^d$ satisfies the conditions:
\begin{itemize}
\item[(f.1)] $f(0)=0$;
\item[(f.2)] the function $f(\cdot)$ is local Lipchitz continuous in a neighborhood of the origin.
\end{itemize}

Since $f$ is local Lipschitz continuous, \cite[Theorem 6.5]{Kai} implies unique existence of solutions of initial value problems
\eqref{Eq_Ex}, $x(0) = x_0$ for $x_0 \in \R^n$. Let $\varphi : I \times \R^d \rightarrow \R^d$, $t \mapsto \varphi(t,x_0)$ denote the solution of \eqref{Eq_Ex}, $x(0) = x_0$, on its maximal interval of existence $I = [0,t_{\max}(x_0))$ with $0 < t_{\max}(x_0) \leq \infty$. We now give the notions of stability of the trivial solution of \eqref{Eq_Ex}.
\begin{definition}\label{DS}
\begin{itemize}
\item[(i)] The trivial solution of \eqref{Eq_Ex} is called \emph{stable} if for any $\varepsilon >0$ there exists $\delta=\delta(\varepsilon)>0$ such that for every $\|x_0\|<\delta$ we have $t_{\max}(x_0) =\infty$ and
\[
\|\varphi(t,x_0)\|< \varepsilon,\quad \forall t\geq 0.
\]
\item[(ii)] The trivial solution is called \emph{asymptotically stable} if it is stable and there exists $\widehat{\delta}> 0$ such that $\lim_{t\to \infty}\varphi(t,x_0)=0$ whenever $\|x_0\|<\widehat\delta$.
\end{itemize}
\end{definition}
\section{Lyapunov direct method for fractional order systems}
In this section, we will establish a Lyapunov candidate function for a fractional-order system to analyze the asymptotic behavior of solutions around the equilibrium points. To do this, we need the following preparatory result which gives an upper bound of the fractional derivative of a composite function.
\begin{theorem}\label{Inequality}
For a given $x_0\in \R^d$, let $u\in \{x_0\}+I^\alpha_{0+} C([0,T],\R^d)$ and $V:\R^d\rightarrow \R$ satisfies the following conditions:
\begin{itemize}
\item[(V.1)] the function $V$ is convex on $\R^d$ and $V(0)=0$;
\item[(V.2)] the function $V$ is differentiable on $\R^d$.
\end{itemize}
Then the following inequality holds for all $t\in [0,T]$:
\begin{equation}\label{Inequality_est}
^{C}D^\alpha_{0+}V(u(t))\leq \langle \nabla V(u(t)),^{C}D^\alpha_{0+}u(t)\rangle,
\end{equation}
where $\nabla V$ is the gradient of the function $V$.
\end{theorem}
\begin{proof}
Due to $u\in \{x_0\}+I^\alpha_{0+} C([0,T],\R^d)$, there exists a function $\psi\in C([0,T],\R^d)$ such that $u=x_0+I^\alpha_{0+} \psi$. From \cite[Proposition 6.4, pp. 479]{Vainikko}, we see that the Caputo fractional derivative $^{C}D^\alpha_{0+}u$ exists and continuous on $[0,T]$. On the other hand, by Theorem \ref{Theorem1}, this derivative has the representation
\begin{equation}\label{eq1}
^{C}D^\alpha_{0+}u(0):=\Gamma(\alpha+1)\gamma,
\end{equation}
where $\gamma=\frac{\psi(0)}{\Gamma(\alpha+1)}$, and
\begin{align}\label{eq2}
\notag ^{C}D^\alpha_{0+}u(t)&=\frac{u(t)-u(0)}{\Gamma(1-\alpha)t^\alpha}\\
&\hspace{0.5cm}+\frac{\alpha}{\Gamma(1-\alpha)}\int_0^t \frac{u(t)-u(\tau)}{(t-\tau)^{\alpha+1}}d\tau,\quad 0<t\leq T.
\end{align}
Using (V.1), (V.2) and by a direct computation, $\lim_{t\to 0}\frac{V(u(t))-V(u(0))}{t^\alpha}=\langle\nabla V(u(0)),\gamma\rangle$. Moreover, from (V.2) and the fact
\[\sup_{0<t\leq T}\left\|\int_{\theta t}^t\frac{u(t)-u(\tau)}{(t-\tau)^{\alpha+1}}\;d\tau\right\|\to 0\qquad \text{as}\quad \theta\to 1,\]
the limit below holds
\[
\sup_{0<t\leq T}\left|\int_{\theta t}^t\frac{V(u(t))-V(u(\tau))}{(t-\tau)^{\alpha+1}}\;d\tau\right|\to 0\qquad \text{as}\quad \theta\to 1,
\]
which together with Theorem \ref{Theorem1} shows that
\begin{equation}\label{eq3}
^{C}D^\alpha_{0+}V(u(0))=\Gamma(\alpha+1)\langle\nabla V(u(0)),\gamma\rangle
\end{equation}
 and for $t\in (0,T]$:
\begin{align}\label{eq4}
\notag ^{C}D^\alpha_{0+}V(u(t))&=\frac{V(u(t))-V(u(0))}{\Gamma (1-\alpha)t^\alpha}\\
&\hspace{0.5cm}+\frac{\alpha}{\Gamma(1-\alpha)}\int_0^t \frac{V(u(t))-V(u(\tau))}{(t-\tau)^{\alpha+1}}d\tau.
\end{align}
From \eqref{eq1} and \eqref{eq3} we have
\begin{equation}\label{eq5}
^{C}D^\alpha_{0+}V(u(0))=\langle\nabla V(u(0)),^{C}D^\alpha_{0+}u(0)\rangle.
\end{equation}
For $0<t\leq T$, using the representations \eqref{eq2} and \eqref{eq4} leads to
\begin{align}
&\notag ^{C}D^\alpha_{0+}V(u(t))-\langle\nabla V(u(t)),^{C}D^\alpha_{0+}u(t)\rangle\\
\notag&\hspace{0.5cm}=\frac{V(u(t))-V(u(0))-\langle\nabla V(u(t)),
u(t)-u(0)\rangle}{\Gamma (1-\alpha)t^\alpha}\\
&\notag\hspace{1cm}+\frac{\alpha}{\Gamma(1-\alpha)}\int_0^t \frac{V(u(t))-V(u(\tau))}{(t-\tau)^{\alpha+1}}d\tau\\
&\hspace{1cm}-\frac{\alpha}{\Gamma(1-\alpha)}\int_0^t\frac{\langle \nabla V(u(t)),u(t)-u(\tau)\rangle}{(t-\tau)^{\alpha+1}}d\tau.\label{eq7}
\end{align}
Because $V$ is convex and differentiable, using \cite[Theorem 25.1, pp. 242]{Rockafellar}, we obtain
\[
V(u(t))-V(u(\tau))-\langle \nabla V(u(t)),u(t)-u(\tau)\rangle\leq 0
\]
for all $0\leq \tau\leq t\leq T$, which together with \eqref{eq5} and \eqref{eq7} implies that
\[
^{C}D^\alpha_{0+}V(u(t))\leq \langle\nabla V(u(t)),^{C}D^\alpha_{0+}u(t)\rangle, \quad \forall t\in [0,T].
\]
The proof is complete.
\end{proof}
\begin{remark}\label{Remark1}
A special case of Theorem \ref{Inequality} when $V(x)=\|x\|^2$ was proven by Aguila-Camacho, Duarte-Mermoud and Gallegos \cite[Lemma 1\;\& Remark 1]{Camacho}. In the case $V$ is convex and differentiable, the inequality \eqref{Inequality_est} was formulated by Chen {\textit{et al.}} \cite[Theorem 1]{Chen}. To obtain the proof of these results, the authors of \cite{Camacho, Chen} required that the function $x$ in Theorem \ref{Inequality} is differentiable. However, in general, the solutions to fractional differential equations are not differentiable. Thus, in our opinion, this assumption is too restrictive which makes the inequality \eqref{Inequality_est} unable to be directly applied to study the asymptotic behavior of solutions to fractional systems. Our result as presented in Theorem 3 now removes this very restrictive assumption.
\end{remark}

We are now in a position to state the main theorem. It is worth noticing that we do not need the assumption of the global existence of solutions to the system \eqref{Eq_Ex}.
\begin{theorem}\label{main_res}
Consider the equation \eqref{Eq_Ex}. Assume there is a function $V:\R^d\rightarrow \R_+$ satisfying
\begin{itemize}
\item[(V.i)] the function $V$ is convex and differentiable on $\R^d$;
\item[(V.ii)] there exist constants $a,b,C_1,C_2,r>0$ such that $$C_1\|x\|^a\leq V(x)\leq C_2 \|x\|^b$$
for all $x\in B_r(0)$;
\item[(V.iii)] there are constants $C_3\geq 0$ and $c\geq b$ such that
$$\langle \nabla V(x),f(x)\rangle\leq -C_3\|x\|^c$$
for all $x\in B_r(0)$.
\end{itemize}
Then,
\begin{itemize}
\item[(a)] the trivial solution of \eqref{Eq_Ex} is stable if $C_3=0$;
\item[(b)] the trivial solution of \eqref{Eq_Ex} is asymptotically stable if $C_3>0$ and $c\geq b$.
\end{itemize}
\end{theorem}
\begin{proof}
From the assumption (f.1), there is a constant $r_1\in (0,r)$ such that $f$ is Lipschitz continuous on $B_{r_1}(0)$. Let $L$ be a Lipschitz constant to $f$ on $B_{r_1}(0)$ and let $F$ denote an extended Lipschitz function of $f$ with the Lipschitz constant $L$, i.e. the function $F:\R^d\rightarrow \R^d$ is Lipschitz continuous with the Lipschitz constant $L$ and $F(x)=f(x)$ for all $x\in B_{r_1}(0)$. Note that this extension always exists, see e.g. \cite[Theorem 2.5]{Heinonen}. For any $\varepsilon\in (0,r_1)$, we choose $\delta=\frac{1}{K}\left(\frac{C_1}{C_2}\right)^{1/b}\varepsilon^{a/b}$, where $K>1$ is large enough to $\delta<\varepsilon$. For any $x_0\in B_\delta(0)$, denote $\tilde{\varphi}(\cdot,x_0)$ the solution to the initial problem
\begin{equation}\label{bt}
\begin{cases}
^{C}D^\alpha_{0+}x(t)=F(x(t)),\quad t>0,\\
x(0)=x_0.
\end{cases}
\end{equation}
Due to \cite[Theorem 2]{Baleanu}, this solution is defined uniquely on the  whole interval $[0,\infty)$. Assume that there is a time $t>0$ such that $\|\tilde{\varphi}(t,x_0)\|=\varepsilon$. Put $t_0:=\inf\{t>0:\|\tilde{\varphi}(t,x_0)\|\geq \varepsilon\}$, then $t_0>0$, $\|\tilde{\varphi}(t_0,x_0)\|=\varepsilon$ and $\|\tilde{\varphi}(t,x_0)\|<\varepsilon$ for all $t\in [0,t_0)$. Hence, $\tilde{\varphi}(\cdot,x_0)$ satisfies the conditions (V.ii) and (V.iii) on the interval $[0,t_0]$.\\

\noindent (a) Now consider the case $C_3=0$. Using Theorem \ref{Inequality}, we have
\[
^{C}D^\alpha_{0+}V(\tilde{\varphi}(t,x_0))\leq \langle\nabla V(\tilde{\varphi}(t,x_0)),F(\tilde{\varphi}(t,x_0))\rangle\leq 0,
\]
for all $t\in [0,t_0]$. Hence, by the comparison lemma \cite[Lemma 10]{Li}, the following estimation holds
\begin{equation}
V(\tilde{\varphi}(t,x_0))\leq V(x_0),\quad \text{for all}\;t\in [0,t_0],
\end{equation}
this combining with (V.i) implies that
\begin{equation}\label{est_eq1}
\|\tilde{\varphi}(t,x_0)\|\leq \left(\frac{C_2}{C_1}\|x_0\|^b\right)^{1/a},\quad\text{for all}\;t\in [0,t_0].
\end{equation}
From \eqref{est_eq1}, we see
\[
\|\tilde{\varphi}(t_0,x_0)\|\leq \left(\frac{C_2}{C_1}\|x_0\|^b\right)^{1/a}\leq \left(\frac{C_2}{C_1}\delta^b\right)^{1/a}<\varepsilon,
\]
a contradiction. Thus, $\|\tilde{\varphi}(t,x_0)\|<\varepsilon$ for all $t\in [0,\infty)$. However, in this case, $\tilde{\varphi}(\cdot,x_0)$ is also a solution to the original equation \eqref{Eq_Ex} with the initial condition $x(0)=x_0$, which shows that the trivial solution to \eqref{Eq_Ex} is stable.\\

\noindent (b) Assume that $C_3>0$. As proved in part (a), we see that the trivial solution to \eqref{Eq_Ex} is stable. Hence, for $\varepsilon>0$ small enough (for example choosing $\varepsilon< r_1$), there exists $\delta>0$ such that every solution $\varphi(t,x_0)$ to \eqref{Eq_Ex} with $\|x_0\|<\delta$ satisfies $\|\varphi(t,x_0)\|<\varepsilon$ for all $t\geq 0$. Moreover, from Theorem 3 and the conditions (V.ii) and (V.iii), we have
\begin{align*}
^{C}D^\alpha_{0+}V(\varphi(t,x_0))&\leq -C_3\|\varphi(t,x_0)\|^c\\
&\leq -\frac{C_3}{C_2^{c/b}}(V(\varphi(t,x_0)))^{c/b},\quad \forall t\geq 0.
\end{align*}
Put $A:=-\frac{C_3}{C_2^{c/b}}$, $p:=\frac{c}{b}$ and consider the following initial value problem 
\begin{equation}\label{tam_1}
\begin{cases}
^{C}D^\alpha_{0+}y(t)=Ay^p(t),\quad t>0,\\
y(0)=y_0>0.
\end{cases}
\end{equation}
Following \cite[Theorem 5.4]{Feng}, the solution $\Phi(\cdot,y_0)$ to \eqref{tam_1} exists on the whole interval $[0,\infty)$ and satisfies $\lim_{t\to \infty}\Phi(t,y_0)=0$. On the other hand, by the comparison lemma \cite[Lemma 10]{Li}, we obtain that
\[
V(\varphi(t,x_0))\leq \Phi(t,V(x_0)),
\quad \forall t\geq 0.
\]
This implies that for any $x_0\in B(0,\delta)\setminus\{0\}$, we have
\[
\lim_{t\to\infty}\|\varphi(t,x_0)\|^a\leq \frac{1}{C_1}\lim_{t\to\infty}V(\varphi(t,x_0))\leq \frac{1}{C_1}\lim_{t\to\infty}\Phi(t,V(x_0))=0.
\]
Note that from the existence and uniqueness of the solutions to \eqref{bt}, if $x_0=0$ then $\varphi(\cdot,0)=0$. So, the trivial solution to the original system \eqref{Eq_Ex} is asymptotically stable. The proof is complete.
\end{proof}
\begin{remark}
Recently, Chen {\emph{et al.}} \cite[Theorem 2, pp. 1072]{Chen} proposed a fractional Lyapunov candidate for fractional systems with the same assumptions as in Theorem \ref{main_res}. However, the proof of their result is incomplete. Indeed, the approach in \cite[Theorem 2]{Chen} is based on the inequality (10) in \cite[Theorem 1]{Chen} and \cite[Theorem 8, pp. 1967]{Li}. As mentioned in Remark \ref{Remark1}, this inequality was established only for differentiable functions. On the other hand, solutions to fractional differential equations are generally not differentiable. Thus, it is impossible to prove \cite[Theorem 2]{Chen} by using the inequality (10) in \cite[Theorem 1]{Chen} as the authors asserted.
\end{remark}
\begin{remark}\label{mistake_1}
Many researchers have proposed fractional Lyapunov functions for fractional-order systems by combining the inequalities \cite[Inequality (16), pp. 2954]{Camacho}, or \cite[Inequality (24), pp. 654]{Camacho_1}, or \cite[Inequality (10), pp. 8]{Chen} with \cite[Theorem 11]{Li}, see e.g. \cite[Theorem 3, pp. 1072]{Chen}, \cite[Theorem 1, pp. 1361]{Aghababa}, \cite[Theorems 2\;\& 3, pp. 684]{Ding}. However, it should be noted that there is a gap in the proof of \cite[Theorem 11]{Li}. Indeed, to prove this theorem, the authors \cite{Li} relied on the following arguments as follows. Consider a function $x:[0,\infty)\rightarrow \R_+$. If $x$ does not satisfy the two conditions as below:\\
{\bf Case 1}: there is a constant $t_1>0$ such that $x(t_1)=0$; and\\
{\bf Case 2:} there exists a positive constant $\varepsilon$ such that $x(t)\geq \varepsilon,\;\forall t\geq 0$.\\
Then $\lim_{t\to \infty}x(t)=0$. Unfortunately, this argument seems incorrect. For a counter example, we consider the function $x(t)=\frac{1}{1+t}+\sin(t)+1,\;\forall t\geq 0$. It is obvious that $x(t)>0,\;\forall t\geq 0. $ Furthermore, there exists the sequence $\{t_k\}_{k=1}^\infty$, where $t_k=-\frac{\pi}{2}+2k\pi,\; k\in \N$, such that $\lim_{k\to\infty}x(t_k)=0$. Hence, there does not exist a parameter $\varepsilon>0$ such that $x(t)\geq \varepsilon,\; \forall t\geq 0$. Thus, the function $x$ does not satisfy both {\bf Case 1} and {\bf Case 2} as above. On the other hand, this function does not tend to zero at infinity.
\end{remark}

Finally, we illustrate the theoretical result by two examples as follows.
\begin{example}
Let the equation
\begin{equation*}
^{C}D^\alpha_{0+}x(t)=-Ax(t),
\end{equation*}
where $A\in\R^{d\times d}$ is a symmetric and positive definite matrix. Choosing the Lyapunov function $V(x)=\langle x,x\rangle$ for all $x\in\R^d$ and using Theorem \ref{main_res} (b), we see that the trivial solution to this equation is asymptotically stable.
\end{example}
\begin{example}
Consider the equation
\begin{equation}\label{ex_1}
^{C}D^\alpha_{0+}x(t)=-x^3(t),\quad \forall t\geq 0.
\end{equation}
It is obvious that the function $f(x)=-x^3$ is local Lipschitz. Choosing the function $V(x)=x^2$ for all $x\in \R$. This function satisfies the conditions (V.i), (V.ii) (with $C_1=C_2=1$ and $a=b=2$), and (V.iii) (with $C_3=2$ and $c=4$). Thus, from Theorem \ref{main_res}(b), the trivial solution to \eqref{ex_1} is asymptotically stable. Fig.1 depicts the trajectory of the solutions $\varphi(\cdot,1)$, $\varphi(\cdot,0.6)$ and $\varphi(\cdot,-0.8)$ to the equation \eqref{ex_1} with $\alpha=0.8$. For a small $\varepsilon_0$ (in this case we choose $\varepsilon_0=0.1$), after $t_0=1000s$ these solutions contain in the interval $[-0.1,0.1]$.
\begin{figure}
\centering
          \includegraphics[width=1.05\textwidth]{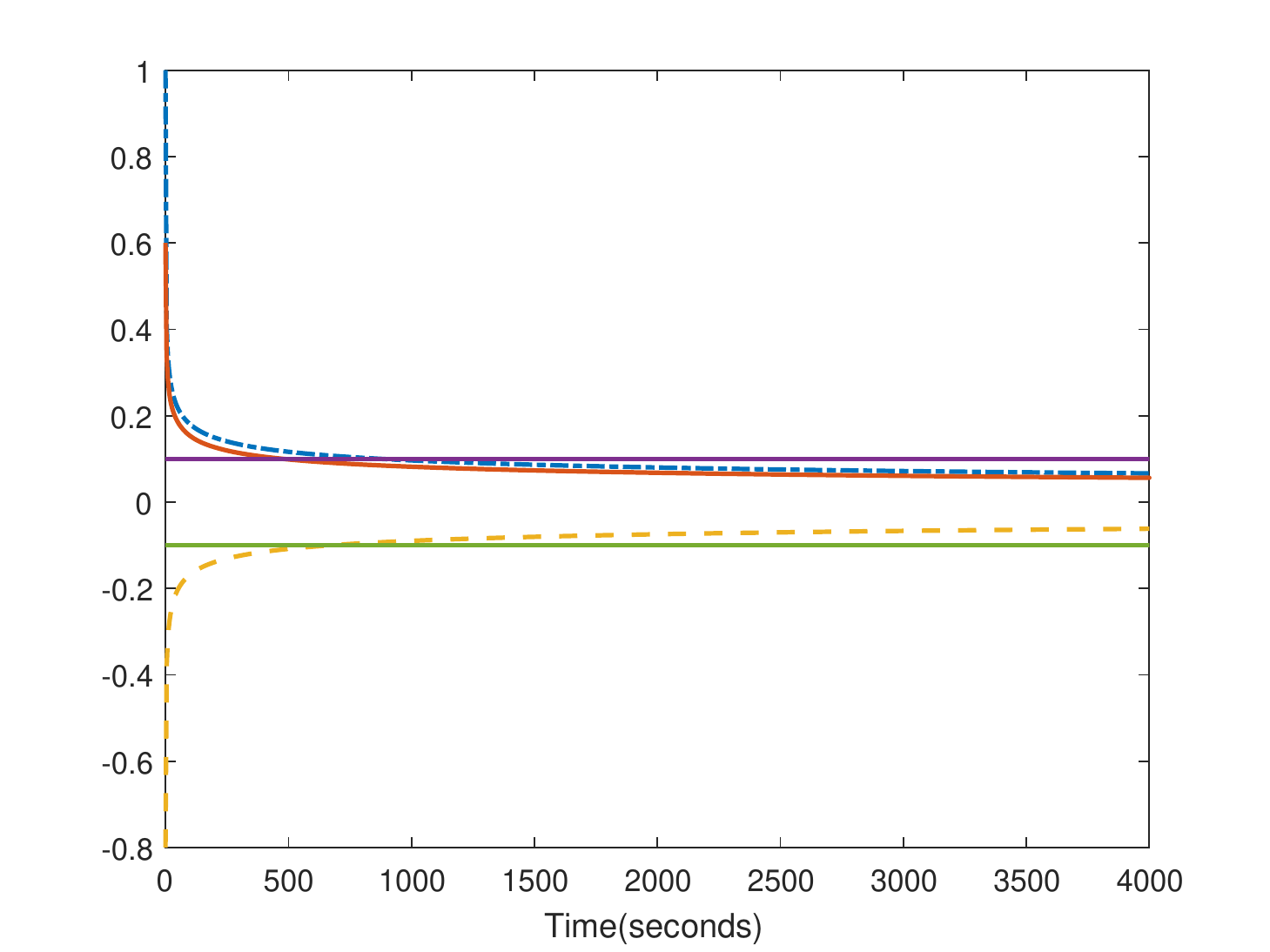}
\caption{Trajectory of the solutions $\varphi(\cdot,1)$, $\varphi(\cdot,0.6)$ and $\varphi(\cdot,-0.8)$ with $\alpha=0.8$.}
\end{figure}
Note that Li \textit{et al.,} \cite[Example 14]{Li} attempted to show that the trivial solution to \eqref{ex_1} is asymptotically stable. Their proof was based on the following statement: Let $x$ be a solution to \eqref{ex_1} with $x(0)\ne 0$. If there is not a constant $\xi>0$ to $x(t)x(0)\geq \xi$ for all $t\geq 0$, then $x(t)\to 0$ as $t\to\infty$, see the lines from -1 to -6, column 2, pp. 1968 in \cite{Li}. Unfortunately, this statement is not correct. For a counterexample, see Remark \ref{mistake_1}. In \cite[Remark 11]{Shen}, the authors revised \cite[Example 14]{Li}. However, their proof is also incomplete because they used \cite[Theorem 11]{Li}. As we have showed in Remark \ref{mistake_1}, the proof of \cite[Theorem 11]{Li} is incomplete. Zhou \textit{et al.} \cite{Zhou} also attempted to prove the asymptotic stability of the trivial solution to \eqref{ex_1}. However, their work was based on an incorrect result, see \cite[Theorem 3.1]{Zhou}.
\end{example}
\section*{Conclusion}
In this paper, we have proposed a rigorous Lyapunov type method to analyze the stability of fractional-order  nonlinear systems. More precisely, we make two main contributions:
\begin{itemize}
\item[$\bullet$] Proving the inequality concerning the fractional derivatives of convex Lyapunov functions without the assumption of the existence of derivative of pseudo-states, see Theorem \ref{Inequality}; and
\item[$\bullet$] Establishing the fractional Lyapunov functions to fractional-order systems without the assumption of the global existence of solutions, see Theorem \ref{main_res}.
\end{itemize}
\section*{Acknowledgement}
The first author is funded by the Vietnam National Foundation for
Science and Technology Development (NAFOSTED) under Grant Number 101.03-2017.01.

\end{document}